\documentclass{amsart}

\usepackage{amsmath}
\usepackage{amsfonts}
\usepackage{amssymb,enumerate}
\usepackage{amsthm}
\usepackage[all]{xy}
\usepackage{hyperref}

\newtheorem{lem}{Lemma}[section]
\newtheorem{cor}[lem]{Corollary}

\newtheorem{thm}[lem]{Theorem}

\newtheorem{Defn}[lem]{Definition}
\newtheorem{Ex}[lem]{Example}
\newtheorem{Question}[lem]{Question}
\newtheorem{Property}[lem]{Property}
\newtheorem{Properties}[lem]{Properties}
\newtheorem{Discussion}[lem]{Remark}
\newtheorem{Construction}[lem]{Construction}
\newtheorem{Notation}[lem]{Notation}
\newtheorem{Fact}[lem]{Fact}
\newtheorem{Assumption}[lem]{Assumption}
\newtheorem{Notationdefinition}[lem]{Definition/Notation}
\newtheorem{Remarkdefinition}[lem]{Remark/Definition}
\newtheorem{Subprops}{}[lem]
\newtheorem{Para}[lem]{}
\newtheorem{Step}{Step}

\newenvironment{defn}{\begin{Defn}\rm}{\end{Defn}}
\newenvironment{ex}{\begin{Ex}\rm}{\end{Ex}}
\newenvironment{question}{\begin{Question}\rm}{\end{Question}}

\newenvironment{notation}{\begin{Notation}\rm}{\end{Notation}}

\newenvironment{assumption}{\begin{Assumption}\rm}{\end{Assumption}}

\newenvironment{disc}{\begin{Discussion}\rm}{\end{Discussion}}

\newcommand{\D}{\mathcal{D}}

\newcommand{\cat}[1]{\mathcal{#1}}

\newcommand{\cata}{\cat{A}}

\newcommand{\catb}{\cat{B}}

\newcommand{\id}{\operatorname{id}}

\newcommand{\rank}{\operatorname{rank}}

\newcommand{\card}[1]{|#1|}

\newcommand{\ext}{\operatorname{Ext}}   
\newcommand{\rhom}{\mathbf{R}\!\operatorname{Hom}}      
\newcommand{\lotimes}{\otimes^{\mathbf{L}}}
\newcommand{\HH}{\operatorname{H}}
\newcommand{\Hom}{\operatorname{Hom}}

\newcommand{\catss}{\mathfrak{S}}
\newcommand{\cats}{\mathfrak{S}}

\newcommand{\im}{\operatorname{Im}}

\newcommand{\Ker}{\operatorname{Ker}}

\newcommand{\ideal}[1]{\mathfrak{#1}}
\newcommand{\m}{\ideal{m}}
\newcommand{\n}{\ideal{n}}

\newcommand{\zz}{\mathbb{Z}}

\newcommand{\xra}{\xrightarrow}

\newcommand{\vf}{\varphi}

\newcommand{\tri}{\trianglelefteq}

\newcommand{\is}[1]{[#1]}
\newcommand{\eq}[1]{[#1]}

\renewcommand{\geq}{\geqslant}
\renewcommand{\leq}{\leqslant}

\renewcommand{\hom}{\Hom}

\newcommand{\ui}{\underline{i}}
\newcommand{\us}{\underline{s}}
\newcommand{\trineq}{\triangleleft}

\numberwithin{equation}{lem}

\begin{document}

\bibliographystyle{amsplain}

\title[The number of semidualizing complexes]
{Lower bounds for the number of semidualizing complexes over a local ring}

\author{Sean Sather-Wagstaff}
\address{Sean Sather-Wagstaff, Department of Mathematics,
NDSU Dept \# 2750,
PO Box 6050,
Fargo, ND 58108-6050 USA}
\email{Sean.Sather-Wagstaff@ndsu.edu}
\urladdr{http://math.ndsu.nodak.edu/faculty/ssatherw/}
\thanks{The author was supported in part by a grant from
the NSA}

\date{\today}



\begin{abstract}
We investigate the set $\catss(R)$ of shift-isomorphism classes of semi-dualizing
$R$-complexes, ordered via
the reflexivity relation, where $R$ is a commutative noetherian local ring.
Specifically, we study the
question of whether $\catss(R)$
has cardinality $2^n$ for some $n$.  
We show that, if there is a 
chain of length $n$ in $\catss(R)$ and if the reflexivity ordering
on $\catss(R)$ is transitive, then $\catss(R)$
has cardinality at least $2^n$, and we 
explicitly describe some of its order-structure.
We also show that, given a local ring homomorphism $\vf\colon R\to S$ of 
finite flat dimension, 
if $R$ and $S$ admit dualizing complexes and if
$\vf$ is not Gorenstein, then
the cardinality of $\catss(S)$ is at least twice the cardinality of $\catss(R)$.
\end{abstract}

\maketitle

\section{Introduction} \label{sec01}

Throughout this work $(R,\m)$ and $(S,\n)$ are commutative noetherian local rings.  

A homologically finite $R$-complex $C$ is \emph{semidualizing} if the natural
homothety morphism $R\to\rhom_R(C,C)$ is an isomorphism
in the derived category $\D(R)$.
(See Section~\ref{sec02}  for background material.) 
Examples of semidualizing $R$-complexes include $R$ itself
and a dualizing $R$-complex when one exists.
The set of shift-isomorphism classes of semidualizing
$R$-complexes is denoted $\catss(R)$, 
and the shift-isomorphism class of a semidualizing $R$-complex $C$ is denoted $\eq C$.

Semidualizing complexes were introduced 
by Avramov and Foxby~\cite{avramov:rhafgd} and Christensen~\cite{christensen:scatac} 
in part to investigate
the homological properties of local ring homomorphisms.
Our interest in these complexes comes from their potential
as tools for answering the composition question for local ring homomorphisms
of finite G-dimension. Unfortunately, 
the utility of the semidualizing $R$-complexes is hampered by the fact that
our 
understanding of $\cats(R)$ is very limited.
For instance, we do not even 
know if the set $\catss(R)$ is finite; see~\cite{christensen:cmafsdm}
for some recent progress.  

We are interested in the following question, motivated by
results from~\cite{frankild:sdcms}, wherein $\card{\catss(R)}$ is the
cardinality of the set $\catss(R)$.

\begin{question} \label{q0101}
If $R$ is a local ring, must we have  $\card{\catss(R)}=2^n$ for some $n\in\mathbb{N}$?
\end{question}

Each semidualizing $R$-complex $C$ gives rise to a
notion of reflexivity 
for homologically finite $R$-complexes.
For instance, each homologically finite $R$-complex of finite projective
dimension is $C$-reflexive.  On the other hand, a semidualizing $R$-complex
$C$ is dualizing if and only if every
homologically finite $R$-complex is $C$-reflexive.  
We order $\catss(R)$ using reflexivity: write $\eq C\tri\eq B$ whenever
$B$ is $C$-reflexive.  This relation  is  reflexive and antisymmetric, but
we do not know if it is transitive in general.
A \emph{chain} in $\catss(R)$ is a sequence $\eq{C_0}\tri\eq{C_1}\tri\cdots\tri\eq{C_n}$,
and
such a chain has \emph{length $n$} if $\eq{C_i}\neq\eq{C_j}$ whenever $i\neq j$.

The main result of this paper, stated next, uses the lengths of chains
in $\cats(R)$ to provide a lower bound of the form $2^n$ on the cardinality
of $\catss(R)$. It is part of Theorem~\ref{thm0402} which also contains the analogous
result for the set of 
isomorphism classes of semidualizing $R$-modules.

\begin{thm} \label{thm0101}
Assume that the reflexivity ordering on $\catss(R)$ is transitive.
If $\catss(R)$ admits a chain of length $n$,
then  $\card{\catss(R)}\geq 2^n$.
\end{thm}

An alternate proof of this result is given in Corollary~\ref{cor0403}.
One advantage of this second method is that we can describe all the
reflexivity relations between these complexes in terms of combinatorial data;
see Theorem~\ref{thm0403}.

Using the ideas from Theorem~\ref{thm0101}, we also prove the following comparison result
which is a special case of Theorem~\ref{thm0401}.

\begin{thm} \label{thm0102}
Let $\vf\colon R\to S$ be a local ring homomorphism  of 
finite flat dimension.  
If $R$ and $S$ admit dualizing complexes and if
$\vf$ is not Gorenstein, then
$\card{\catss(S)}\geq 2\card{\catss(R)}$.
\end{thm}

\section{Complexes and local ring homomorphisms}\label{sec02}

This section contains definitions and background material for use in the sequel.

\begin{defn} \label{notn01}
An \emph{$R$-complex} is a sequence of 
$R$-module homomorphisms 
$$X =\cdots\xra{\partial^X_{n+1}}X_n\xra{\partial^X_n}
X_{n-1}\xra{\partial^X_{n-1}}\cdots$$
such that $ \partial^X_{n-1}\partial^X_{n}=0$ for each integer $n$. The
$n$th \emph{homology module} of $X$ is
$\HH_n(X):=\Ker(\partial^X_{n})/\im(\partial^X_{n+1})$.
The complex $X$ 
is \emph{homologically bounded} when $\HH_n(X)=0$ for $|n|\gg 0$.
It is  \emph{homologically finite} if the
$R$-module $\oplus_{n\in\zz}\HH_n(X)$ is finitely generated.
We frequently identify $R$-modules  with $R$-complexes concentrated in degree 0.
\end{defn}

\begin{notation}
We work in the derived category $\D(R)$.
References 
on the subject include~\cite{gelfand:moha,hartshorne:rad,verdier:cd,verdier:1};
see also~\cite{sather:bnsc}.
Given two $R$-complexes $X$ and $Y$,
the derived homomorphism and tensor product complexes
are denoted $\rhom_R(X,Y)$ and $X\lotimes_R Y$.
Isomorphisms in $\D(R)$ are identified by the symbol $\simeq$,
and isomorphisms up to shift are identified by $\sim$. 
\end{notation}

\begin{defn} \label{defn0201}
The $n$th \emph{Bass number} of $R$ is
$\mu^n_R(R)=\rank_{R/\m}(\ext^n_R(R/\m,R))$, 
and the
\emph{Bass series} of $R$ is the power series
$I^R_R(t)=\sum_{n=0}^{\infty}\mu^n_R(R)r^n$.

Let $\vf\colon R\to S$ be a local ring homomorphism of finite flat dimension,
that is, such that  $S$ admits a bounded resolution by flat $R$-modules.
The \emph{Bass series} of $\vf$ is a formal Laurent series $I_{\vf}(t)$
with nonnegative integer coefficients such that
$I^S_S(t)=I_{\vf}(t)I^R_R(t)$; see~\cite[(5.1)]{avramov:bsolrhoffd} for the
existence of $I_{\vf}(t)$. The homomorphism $\vf$ is
\emph{Gorenstein} at $\n$ if $I_{\vf}(t)=t^d$ for some integer $d$. 
\end{defn}

\begin{ex} 
Let $\vf\colon R\to S$ be a local ring homomorphism of finite flat dimension.
When  $\vf$ is flat, it is  Gorenstein if and only if
the closed fibre $S/\m S$ is Gorenstein. 
Also, if $\vf$ is surjective
with kernel
generated by an $R$-sequence, then it is  Gorenstein.
\end{ex}

Semidualizing complexes, defined next, are our main objects of study.

\begin{defn} \label{notn06}
A homologically finite $R$-complex $C$
is \emph{semidualizing} if the natural homothety morphism
$\chi^R_C\colon R\to \rhom_R(C,C)$
is an isomorphism in $\D(R)$.
An $R$-complex $D$ is \emph{dualizing} if it is semidualizing
and has finite injective dimension.
Let $\cats(R)$ denote the set of shift-isomorphism classes of semidualizing
$R$-complexes, and let $\is C$ denote the shift-isomorphism class of
a semidualizing $R$-complex $C$.

When $C$ is a finitely generated $R$-module, it is semidualizing if and only if
$\ext^{\geq 1}_R(C,C)=0$ and the natural homothety map
$R\to \hom_R(C,C)$
is an isomorphism. 
Let $\cats_0(R)$ denote the set of isomorphism classes of semidualizing
$R$-modules, and let $\is C$ denote the isomorphism class of
a semidualizing $R$-module $C$.
The natural identification of an $R$-module with an $R$-complex
concentrated in degree 0 provides a natural inclusion
$\cats_0(R)\subseteq\cats(R)$.
\end{defn}

\begin{disc} \label{sp03}
Let $\vf\colon R\to S$ be a local ring homomorphism
of finite flat dimension, and fix semidualizing
$R$-complexes $B,C$. The complex $S\lotimes_R C$ is
semidualizing for $S$ by~\cite[(5.7)]{christensen:scatac}.
The complex $S\lotimes_R C$ is
dualizing for $S$ if and only if
$C$ is dualizing for $R$ and $\vf$ is Gorenstein
by~\cite[(5.1)]{avramov:lgh}.
We have $S\lotimes_R B\simeq S\lotimes_R C$ in $\D(S)$
if and only if $B\simeq C$ in $\D(R)$ by~\cite[(1.10)]{frankild:rrhffd}.
Hence, the function $\cats(\vf)\colon\cats(R)\to\cats(S)$ given by
$\eq{C}\mapsto\eq{S\lotimes_R C}$ is well-defined and injective.
\end{disc}

The  next definition is due to Christensen~\cite{christensen:scatac}
and Hartshorne~\cite{hartshorne:rad} and 
will be used primarily to compare semidualizing complexes.

\begin{defn} \label{notn08}
Let $C$ be a semidualizing $R$-complex.
A homologically finite $R$-complex $X$ is \emph{$C$-reflexive} 
when 
the $R$-complex $\rhom_R(X,C)$ is homologically finite, and
the natural biduality morphism $\delta^C_X\colon X\to\rhom_R(\rhom_R(X,C),C)$
is an isomorphism in $\D(R)$.
Define an order on $\catss(R)$ by writing $\eq{C}\tri\eq{B}$ when
$B$ is $C$-reflexive. Also, write $\eq{C}\trineq\eq{B}$ when
$\eq{C}\tri\eq{B}$ and $\eq{C}\neq\eq{B}$. 
For each $\eq C\in\cats(R)$ set
$\cats_C(R)=\{\eq B\in\cats(R)\mid\eq C\tri\eq B\}$.
\end{defn}

\begin{disc}  \label{disc0201}
Let $A$, $B$ and $C$ be semidualizing $R$-complexes.

1. If $B$ is $C$-reflexive, then $\rhom_R(B,C)$ is
semidualizing and $C$-reflexive by~\cite[(2.12)]{christensen:scatac}.
Thus, the map $\Phi_C\colon\cats_C(R)\to\cats_C(R)$ given by $\eq{B}\mapsto
\eq{\rhom_R(B,C)}$ is  well-defined.  By definition, this map is 
also an involution (i.e., $\Phi_C^2=\id_{\cats^{}_C(R)}$) and hence it is bijective.
From~\cite[(3.9)]{frankild:rrhffd} we know that $\Phi_C$ is reverses the reflexivity ordering:
if $\eq A,\eq B\in\catss_C(R)$, then
$\eq A\tri\eq B$ if and only if $\Phi_C(\is B)\tri \Phi_C(\is A)$, that is, 
if and only if
$\eq{\rhom_R(B,C)}\tri\eq{\rhom_R(A,C)}$.

2. Assume that $C$ is a semidualizing $R$-module.
Using~\cite[(3.5)]{frankild:rrhffd}
we see that,
if $B$ is $C$-reflexive, then $B$ is isomorphic up to shift with a semidualizing $R$-module,
and hence so is $\rhom_R(B,C)$.
In particular, we have $\catss_C(R)\subseteq\catss_0(R)$.

3. If $D$ is a dualizing $R$-complex, then $\eq D\tri\eq C$
by~\cite[(V.2.1)]{hartshorne:rad},
i.e., we have $\cats_D(R)=\cats(R)$.

4. Let $X$ be an $R$-complex such that $\HH_i(X)$ is finitely generated for each $i$
and $\HH_i(X)=0$ for $i\ll 0$.  If  $C\lotimes_R C\lotimes_R X$ is
semidualizing, then $C\sim R$ by~\cite[(3.2)]{frankild:sdcms}.

5. By~\cite[(3.3)]{gerko:sdc}, given a chain
$\eq{C_n}\trineq\eq{C_{n-1}} \trineq\cdots \trineq\eq{C_0}$ in $\catss (R)$, one has
\begin{align*}
C_n&\simeq
C_0\lotimes_R\rhom_R(C_0,C_1)\lotimes_R\cdots\lotimes_R\rhom_R(C_{n-1},C_n).
\end{align*}
\end{disc}

\begin{disc}  \label{disc0202}
Let $\vf\colon R\to S$ be a local ring homomorphism of finite flat dimension.
The map $\cats(\vf)\colon\cats(R)\to\cats(S)$ from Remark~\ref{sp03}
respects the reflexivity orderings perfectly by~\cite[(4.8)]{frankild:rrhffd}:
if $\is B,\is C\in\catss(R)$, then 
$\eq C\tri\eq B$ if and only if 
$\cats(\vf)(\eq C)\tri \cats(\vf)(\eq B)$
that is, if and only if 
$\eq{S\lotimes_R C}\tri\eq{S\lotimes_R B}$.
\end{disc}

The next definition is due to 
Avramov and Foxby~\cite{avramov:rhafgd}
and Christensen~\cite{christensen:scatac}.

\begin{defn} \label{notn99}
Let $C$ be a semidualizing $R$-complex.
A homologically bounded $R$-complex $X$ is in the 
\emph{Bass class} with respect to $C$, denoted $\catb_C(R)$,
when 
the $R$-complex $\rhom_R(C,X)$ is homologically bounded and
the natural evaluation morphism $\xi^C_X\colon C\lotimes_R\rhom_R(C,X)\to X$
is an isomorphism in $\D(R)$.
A homologically bounded $R$-complex $X$ is in the 
\emph{Auslander class} with respect to $C$, denoted $\cata_C(R)$,
when 
the $R$-complex $C\lotimes_RX$ is homologically bounded and
the natural morphism $\gamma^C_X\colon X\to \rhom_R(C,C\lotimes_RX)$
is an isomorphism in $\D(R)$.
\end{defn}

\begin{disc}  \label{disc9801}
Let $B$ and $C$ be semidualizing $R$-complexes.
Then $B\in \catb_C(R)$
if and only if $\eq B\tri\eq C$; see~\cite[(1.3)]{frankild:rbsc}.
One has $B\in\cata_C(R)$ if and only if $B\lotimes_RC$ is a semidualizing $R$-complex
by~\cite[(4.8)]{frankild:rbsc}.
\end{disc}

\section{Bounding the number of elements in $\catss(R)$} \label{sec03}

This section contains
the proofs of Theorems~\ref{thm0101}
and~\ref{thm0102} from the introduction.
We begin with two lemmas.

\begin{lem} \label{lem0401}
Let $A$, $B$ and $C$ be semidualizing $R$-complexes
such that 
$B$ and $C$ are $A$-reflexive and
$B$ is $C$-reflexive. If $C\not\sim A$,
then $\rhom_R(B,A)$ is not $C$-reflexive.
\end{lem}

\begin{proof}
Remark~\ref{disc0201}.1 implies that 
$\rhom_R(B,A)$ and $\rhom_R(C,A)$
are semidualizing $R$-complexes and that
$\rhom_R(C,A)$ is $\rhom_R(B,A)$-reflexive.
Remark~\ref{disc0201}.5  
provides the first isomorphism in the next sequence
\begin{align*}
\rhom_R(B,A)
&\simeq  \rhom_R(C,A)\lotimes_R \rhom_R(\rhom_R(C,A),\rhom_R(B,A))\\
&\simeq  \rhom_R(C,A)\lotimes_R \rhom_R(\rhom_R(C,A)\lotimes_RB,A)\\
&\simeq  \rhom_R(C,A)\lotimes_R \rhom_R(B,\rhom_R(\rhom_R(C,A),A))\\
&\simeq \rhom_R(C,A)\lotimes_R \rhom_R(B,C).
\end{align*}
The second and third isomorphisms are Hom-tensor adjointness,
and the fourth isomorphism comes from the fact that $C$ is $A$-reflexive.

Set $X=\rhom_R(B,C)\lotimes_R \rhom_R(\rhom_R(B,A),C)$,
and suppose that the complex $\rhom_R(B,A)$ is $C$-reflexive.
Remark~\ref{disc0201}.5 explains the first isomorphism in the next sequence,
and the second isomorphism  is from the previous display
\begin{align*}
C
&\simeq \rhom_R(B,A)\lotimes_R \rhom_R(\rhom_R(B,A),C)\\
&\simeq \rhom_R(C,A)\lotimes_R \rhom_R(B,C)\lotimes_R \rhom_R(\rhom_R(B,A),C)\\
&\simeq \rhom_R(C,A)\lotimes_R X.
\end{align*}
Similarly, this yields the next sequence
\begin{align*}
A
&\simeq \rhom_R(C,A)\lotimes_R C\\
&\simeq \rhom_R(C,A)\lotimes_R 
\rhom_R(C,A)\lotimes_R X.
\end{align*}
It follows from Remark~\ref{disc0201}.4 that
$\rhom_R(C,A)\sim R$ and hence
$$C\simeq\rhom_R(\rhom_R(C,A),A)\sim\rhom_R(R,A)\simeq A$$
since $C$ is $A$-reflexive.
This  contradicts the assumption $C\not\sim A$.
\end{proof}

Note that the hypothesis $\cats_C(R)\subseteq\cats_A(R)$ from the next result
is satisfied when either $A$ is dualizing for $R$ or the reflexivity ordering on
$\cats(R)$ is transitive.

\begin{lem} \label{prop0401}
Let $A$ and $C$ be semidualizing $R$-complexes
such that $C$ is $A$-reflexive and $C\not\sim A$.
Assume that $\cats_C(R)\subseteq\cats_A(R)$, e.g., that the reflexivity ordering
on $\catss(R)$ is transitive.
The injection 
$\Phi_A\colon\cats_A(R)\to\cats_A(R)$ given by
$\eq{B}\mapsto\eq{\rhom_R(B,A)}$
maps
$\catss_C(R)$ into $\catss_A(R)\smallsetminus\catss_C(R)$.
In particular $\card{\catss_A(R)}\geq 2\card{\catss_C(R)}$.
\end{lem}

\begin{proof}
The first conclusion is a reformulation of
Lemma~\ref{lem0401};
see also Remark~\ref{disc0201}.1. 
For the second conclusion, note that
$\Phi_A$ is injective by Remark~\ref{disc0201}.1,
so $\Phi_A(\cats_C(R))$ and $\cats_C(R)$
have the same cardinality.
Since $\Phi_A(\cats_C(R))\subset\catss_A(R)\smallsetminus\catss_C(R)$,
we conclude that $\catss_C(R)$ and $\Phi_A(\cats_C(R))$ are disjoint subsets of
$\catss_A(R)$ such that $\card{\catss_C(R)}=\card{\Phi_A(\cats_C(R))}$.
The second conclusion now follows. 
\end{proof}

The next result contains Theorem~\ref{thm0101} from the introduction.

\begin{thm} \label{thm0402}
Assume that $\catss(R)$ admits a chain
$\eq{C_n}\trineq\eq{C_{n-1}} \trineq\cdots \trineq\eq{C_0}$ such that
$\catss_{C_0}(R)\subseteq\cdots \subseteq\catss_{C_{n-1}}(R) \subseteq\catss_{C_n}(R)$.
Then  $\card{\catss(R)}\geq 2^n$.
If $\eq{C_n}\in\catss_0(R)$,
then  $\card{\catss_0(R)}\geq 2^n$.
\end{thm}

\begin{proof}
For the first statement,
we show by induction on $j$ that 
$\card{\catss_{C_{j}}(R)}\geq 2^j$.
For $j=0$ this is straightforward.
For the inductive step assume that $\card{\catss_{C_{j}}(R)}\geq 2^{j}$.
Lemma~\ref{prop0401} implies that 
$\card{\catss_{C_{j+1}}(R)}\geq2\card{\catss_{C_{j}}(R)}\geq 2^{j+1}$
as desired.

When
$\is{C_n}\in\catss_0(R)$, we have $\catss_{C_j}(R)\subseteq\catss_{C_n}(R)\subseteq\catss_0(R)$
for $j=0,\ldots,n$
by Remark~\ref{disc0201}.2.
Thus, the second statement is proved like the first statement. 
\end{proof}

We next provide an explicit description of the
$2^n$ semidualizing complexes that are guaranteed to exist by Theorem~\ref{thm0402}.
A  second description
(in the case $C_0\simeq R$) is given in Theorem~\ref{thm0403}.

\begin{disc} \label{d1101}
Assume that $\catss(R)$ admits a chain
$\eq{C_n}\trineq\eq{C_{n-1}} \trineq\cdots \trineq\eq{C_0}$ such that
$\catss_{C_0}(R)\subseteq\cdots \subseteq\catss_{C_{n-1}}(R) \subseteq\catss_{C_n}(R)$.
Given $R$-complexes $C$ and $B$, set
$C^{\dagger_B}=\rhom_R(C,B)$.

The next diagram shows the steps $n=0,1,2,3$ of the induction argument
in the proof of Theorem~\ref{thm0402}, with some of the reflexivity relations indicated with edges:
$$
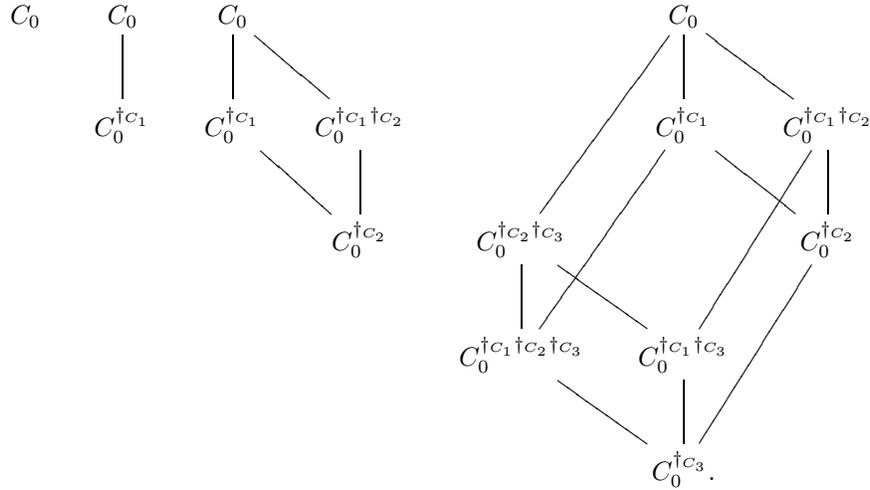
\xymatrix@C=5mm{
C_0&
C_0 \ar@{-}[d] 
& C_0 \ar@{-}[d]\ar@{-}[rd]
&&& C_0 \ar@{-}[d]\ar@{-}[rd]\ar@{-}[ldd] 
\\
&C_0^{\dagger_{C_1}}
&C_0^{\dagger_{C_1}}\ar@{-}[rd]
&C_0^{\dagger_{C_1}\dagger_{C_2}}\ar@{-}[d]
&&C_0^{\dagger_{C_1}}\ar@{-}[rd]\ar@{-}[ldd]
&C_0^{\dagger_{C_1}\dagger_{C_2}}\ar@{-}[d]\ar@{-}[ldd]
\\
&&&C_0^{\dagger_{C_2}}
&C_0^{\dagger_{C_2}\dagger_{C_3}}\ar@{-}[d]\ar@{-}[rd]
&&C_0^{\dagger_{C_2}}\ar@{-}[ldd]
\\
&&&&C_0^{\dagger_{C_1}\dagger_{C_2}\dagger_{C_3}}\ar@{-}[rd]
&C_0^{\dagger_{C_1}\dagger_{C_3}}\ar@{-}[d]
\\
&&&&&C_0^{\dagger_{C_3}}.
}$$

More generally, for each sequence of integers
$\ui=\{i_1,\ldots,i_j\}$ such that $j\geq 0$ and $1\leq i_1<\cdots<i_j\leq n$,
the $R$-complex 
$C_{\ui}=C_0^{\dagger_{C_{i_1}}\dagger_{C_{i_2}}\dagger_{C_{i_3}}\cdots\dagger_{C_{i_j}}}$
is semidualizing. 
(When
$j=0$ we have $C_{\ui}=C_\emptyset= C_0$.)
The   classes  $\eq{C_{\ui}}$
are parametrized by the allowable sequences $\ui$,
of which there are exactly $2^n$. These are precisely the $2^n$
classes constructed in the proof of Theorem~\ref{thm0402}.

If the $C_i$ are all modules, then we may replace $\rhom$ with $\Hom$
to obtain a description of the semidualizing modules that the theorem
guarantees to exist.
\end{disc}

The next result  contains
Theorem~\ref{thm0102} from the introduction.

\begin{thm} \label{thm0401}
Let $\vf\colon R\to S$ be a local ring homomorphism  of 
finite flat dimension.  Assume that $\catss(R)$ has a unique minimal element
$\eq{A}$. (For instance, this holds when $R$ admits a dualizing complex.)
If $S$ admits a dualizing complex $D^S$ and if
$\vf$ is not Gorenstein at $\n$, then
$\card{\catss(S)}\geq 2\card{\catss(R)}$.
\end{thm}

\begin{proof}
Let $\catss(\vf)\colon\catss(R)\to\catss(S)$ be the induced map from
Remark~\ref{sp03}. Our assumption on $A$ implies that $\cats(R)=\cats_A(R)$.
Remark~\ref{disc0202} provides the first containment in the next sequence
while Remark~\ref{disc0201}.3 explains the last equality
$$\cats(\vf)(\cats(R))=\cats(\vf)(\cats_A(R))\subseteq
\catss_{S\lotimes_R A}(S)
\subseteq\catss(S)=\cats_{D^S}(S).$$  
Since $\vf$ is not Gorenstein, Remark~\ref{sp03} implies
$D^S\not\sim S\lotimes_R A$. 
The injectivity of $\cats(\vf)$ explains the first inequality in the next sequence
$$2\card{\catss(R)}=2\card{\catss_A(R)}\leq
2\card{\catss_{S\lotimes_R A}(S)}\leq\card{\catss_{D^S}(S)}=\card{\catss(S)}
$$
while the second inequality is from Lemma~\ref{prop0401}.
\end{proof}



\section{Structure on the Set of Semidualizing $R$-complexes} \label{sec99}

This section is devoted to providing a second description of the $2^n$ semidualizing
complexes that are guaranteed to exist by Theorem~\ref{thm0402}.
This description is contained in Theorems~\ref{thm9901} and~\ref{thm0403}.
It says, in particular, that these complexes form a dual version of the ``LCM lattice''
on $n$ formal letters. Note that the version for semidualizing modules
has the same form with derived functors replaced by
non-derived functors; hence, we do not state it explicitly.
We begin with some notation for use throughout the section and five
supporting lemmas.

\begin{assumption} \label{ass}
Assume throughout this section that 
$\catss(R)$ admits a chain 
$\eq{C_n}\trineq\eq{C_{n-1}} \trineq\cdots \trineq\eq{C_0}$.
For $i=1,\ldots,n$ set $B_i=\rhom_R(C_{i-1},C_i)$.
Set $B_{\emptyset}=C_0$.
For each sequence $\ui=\{i_1,\ldots,i_j\}$ such that $1\leq i_1<\cdots<i_j\leq n$
and $j\geq 1$,
set $B_{\ui}=B_{i_1}\lotimes_R\cdots\lotimes_RB_{i_j}$.
(When $j=1$, we have $B_{\ui}=B_{\{i_1\}}=B_{i_1}$.)
\end{assumption}

\begin{lem} \label{lem9901}
Under the hypotheses of Assumption~\ref{ass},
if $1\leq i<j\leq n$, then
$$\rhom_R(C_{i-1},C_{j-1})\lotimes_R
B_j
\simeq\rhom_R(C_{i-1},C_{j}).
$$
\end{lem}

\begin{proof}
Consider the following sequence of isomorphisms:
\begin{align*}
\rhom_R(B_j,\rhom_R(C_{i-1},C_{j}))\!
&\simeq\rhom_R(\rhom_R(C_{j-1},C_{j}),\!\rhom_R(C_{i-1},C_{j}))\\
&\simeq\rhom_R(\rhom_R(C_{j-1},C_{j})\lotimes_RC_{i-1},C_{j}) \\
&\simeq\rhom_R(C_{i-1},\!\rhom_R(\rhom_R(C_{j-1},C_{j}),C_{j})) \\
&\simeq\rhom_R(C_{i-1},C_{j-1}).
\end{align*}
The first two isomorphisms are from Hom-tensor adjointness
and the commutativity of tensor product.
The third isomorphism is from the condition $\eq{C_{j}}\tri\eq{C_{j-1}}$.
Since the $R$-complexes $B_j$, 
$\rhom_R(C_{i-1},C_{j})$ and $\rhom_R(C_{i-1},C_{j-1})$ are semidualizing,
it follows from~\cite[(1.3)]{frankild:rbsc} that 
$\eq{\rhom_R(C_{i-1},C_{j})}\tri\eq{B_j}$
so the first isomorphism in the next sequence is from
Remark~\ref{disc0201}.5:
\begin{align*}
\rhom_R(C_{i-1},C_{j})
&\simeq B_j\lotimes_R
\rhom_R(B_j,\rhom_R(C_{i-1},C_{j}))\\
&\simeq B_j\lotimes_R
\rhom_R(C_{i-1},C_{j-1})
\end{align*}
The second isomorphism is from the first displayed sequence in this proof.
\end{proof}

\begin{lem} \label{lem9902}
Under  the hypotheses of Assumption~\ref{ass},
if $1\leq i< i+p\leq n$, then
$$B_{\{i,i+1,\ldots,i+p\}}\simeq \rhom_R(C_{i-1},C_{i+p}).$$
\end{lem}

\begin{proof}
We proceed by induction on $p$.
The base case is when $p=1$.
Setting $j=i+1$ in Lemma~\ref{lem9901}
yields the first isomorphism in the next sequence
\begin{align*}
\rhom_R(C_{i-1},C_{i+1})
&\simeq B_{i+1}\lotimes_R
\rhom_R(C_{i-1},C_{i})
\simeq B_{i+1}\lotimes_RB_{i}
\simeq B_{\{i,i+1\}}.
\end{align*}
The  remaining isormorphisms are by definition.

For the inductive step, assume that
$B_{\{i,i+1,\ldots,i+p-1\}}\simeq \rhom_R(C_{i-1},C_{i+p-1})$.
This assmption explains the second isomorphism in the next sequence
\begin{align*}
B_{\{i,i+1,\ldots,i+p\}}
&\simeq B_{i}\lotimes_R\cdots\lotimes_RB_{i+p-1}\lotimes_RB_{i+p}\\
&\simeq \rhom_R(C_{i-1},C_{i+p-1})\lotimes_RB_{i+p}\\
&\simeq \rhom_R(C_{i-1},C_{i+p}).
\end{align*}
The first and third isomorphisms are by definition,
and the fourth isomorphism is from Lemma~\ref{lem9901}.
\end{proof}

\begin{lem} \label{lem9903}
Assume that 
$\catss_{C_{i}}(R)\subseteq\catss_{C_{j-1}}(R)$.
Under  the hypotheses of Assumption~\ref{ass},
if $1\leq i<j-1\leq n-1$, then
$$B_{\{i,j\}}\simeq 
\rhom_R(\rhom_R(B_i,C_{j-1}),C_{j})).$$
(In the notation from Remark~\ref{d1101}, this reads
$B_{\{i,j\}}
\simeq B_{i}^{\dagger_{C_{j-1}}\dagger_{C_{j}}}\simeq 
C_{i-1}^{\dagger_{C_{i}}\dagger_{C_{j-1}}\dagger_{C_{j}}}$.)
\end{lem}

\begin{proof}
The proof is similar to that for Lemma~\ref{lem9901}. The main point is the following
sequence of isomorphisms wherein
the last two isomorphisms are from Hom-tensor adjointness
and the commutativity of tensor product:
\begin{align*}
B_i
&\simeq
\rhom_R(\rhom_R(B_i,C_{j-1}),C_{j-1})\\
&\simeq
\rhom_R(\rhom_R(B_i,C_{j-1}),\rhom_R(B_j,C_{j}))\\
&\simeq
\rhom_R(B_j\lotimes_R\rhom_R(B_i,C_{j-1}),C_{j})\\
&\simeq
\rhom_R(B_j,\rhom_R(\rhom_R(B_i,C_{j-1}),C_{j})).
\end{align*}
The first isomorphism is from the chain
$\eq{C_{j-1}}\tri\eq{C_i}\tri\eq{\rhom_R(C_{i-1},C_i)}=\eq{B_i}$, 
using the containment 
$\catss_{C_{i}}(R)\subseteq\catss_{C_{j-1}}(R)$;
see Remark~\ref{disc0201}.1.
The second isomorphism is from the condition $\eq{C_{j}}\tri\eq{C_{j-1}}$
and the definition of $B_j$. 
\end{proof}

\begin{lem} \label{lem9904}
Assume that 
$\catss_{C_1}(R)\subseteq\cdots \subseteq\catss_{C_{n-1}}(R) \subseteq\catss_{C_n}(R)$.
Under  the hypotheses of Assumption~\ref{ass},
each of the $R$-complexes $B_{\ui}$ is semidualizing.
\end{lem}

\begin{proof}
We proceed  by induction on $n$, the length of the chain in $\catss (R)$.
The base cases $n=0,1$ are routine.
Indeed, when $n=0$, we have only one $B_{\ui}$, namely $B_{\emptyset}=C_0$.
When $n=1$, we have only two $B_{\ui}$, namely
$B_{\{1\}}=B_1$ and $B_{\emptyset}=C_0$.

Now, assume that $n\geq 2$ and that the result holds for chains of length $n-1$.
Also, assume without loss of generality that $\ui\neq\emptyset$.
We have two cases.

Case 1:
If $i_2=i_1+1$, then the first isomorphism in the next sequence is by definition,
and the second isomorphism is from Lemma~\ref{lem9902}:
\begin{align*}
B_{\ui}
&\simeq B_{i_1}\lotimes_R B_{i_1+1}\lotimes_R B_{i_3}\lotimes_R\cdots\lotimes_R B_{i_j}\\
&\simeq \rhom_R(C_{i_1-1},C_{i_1+1})\lotimes_R B_{i_3}\lotimes_R\cdots\lotimes_R B_{i_j}
\end{align*}
Applying our induction hypothesis to 
the chain $\eq{C_n}\trineq \cdots \trineq \eq{C_{i_1+1}} \trineq \eq{C_{i_1-1}} \trineq\cdots \trineq 
\eq{C_0}$,
we conclude that the tensor product in the final line of this sequence is semidualizing.
Hence $B_{\ui}$ is semidualizing in this case.

Case 2: If $i_2>i_1+1$,
then the first isomorphism in the next sequence is by definition,
and the second isomorphism is from Lemma~\ref{lem9903}:
\begin{align*}
B_{\ui}
&\simeq B_{i_1}\lotimes_R B_{i_2}\lotimes_R B_{i_3}\lotimes_R\cdots\lotimes_R B_{i_j}\\
&\simeq \rhom_R(\rhom_R(\rhom_R(C_{i_1-1},C_{i_1}),C_{i_2-1}),C_{i_2}))
\lotimes_R B_{i_3}\lotimes_R\cdots\lotimes_R B_{i_j}
\end{align*}
Applying our induction hypothesis to 
the chain 
$$\eq{C_n}\trineq \cdots \trineq \eq{C_{i_2}} \trineq 
\eq{\rhom_R(\rhom_R(C_{i_1-1},C_{i_1}),C_{i_2-1})}$$
we conclude again that $B_{\ui}$ is semidualizing in this case.
\end{proof}

\begin{lem} \label{lem9905}
Assume that 
$\catss_{C_1}(R)\subseteq\cdots \subseteq\catss_{C_{n-1}}(R) \subseteq\catss_{C_n}(R)$.
Under  the hypotheses of Assumption~\ref{ass},
if $\us=\{s_1,\ldots,s_1\}$ is a sequence of integers such that $1\leq s_1<\cdots<s_t\leq n$
and
$\ui\supseteq \us$, then $\eq{B_{\ui}}\tri \eq{B_{\us}}$ and
$\rhom_R(B_{\us},B_{\ui})\simeq B_{\ui\smallsetminus\us}$.
\end{lem}

\begin{proof}
By Lemma~\ref{lem9904}, the $R$-complexes
$B_{\ui}$, $B_{\us}$ and $B_{\ui\smallsetminus\us}$ are  semidualizing.
By definition, we have $B_{\us}\simeq B_{\ui}\lotimes_RB_{\ui\smallsetminus\us}$,
so the result follows from~\cite[(3.5)]{gerko:sdc}.
\end{proof}

The following theorem compares the complexes constructed in this section to those from the previous section.

\begin{thm} \label{thm9901}
Assume that 
$\catss_{C_1}(R)\subseteq\cdots \subseteq\catss_{C_{n-1}}(R) \subseteq\catss_{C_n}(R)$.
Under  the hypotheses of Assumption~\ref{ass},
if $C_0\simeq R$, then
the  $R$-complexes $B_{\ui}$ 
are precisely the $2^n$ complexes constructed in
Theorem~\ref{thm0402}. 
\end{thm}

\begin{proof}
We proceed  by induction on $n$, the length of the chain in $\catss (R)$.
The base case $n=0$ is straightforward: The only $B_{\ui}$ is
$B_{\emptyset}=C_0\simeq R$, which is the only module constructed in the proof of
Theorem~\ref{thm0402}.

For the inductive step, assume that $n\geq 1$ and that the result holds for
the chain 
$\eq{C_{n-1}}\trineq\cdots \trineq\eq{C_0}$.
Then the  $R$-complexes $B_{\ui}$ with $i_j\leq n-1$ 
are precisely the $2^{n-1}$ complexes constructed in
Theorem~\ref{thm0402} for this chain.

The remaining $2^{n-1}$ semidualizing $R$-complexes constructed in
Theorem~\ref{thm0402} (according to the proof) for
the chain 
$\eq{C_{n}}\trineq\eq{C_{n-1}} \trineq\cdots \trineq\eq{C_0}$
are then of the form $\rhom_R(B_{\ui},C_n)$. Thus, it remains to show that 
each of these complexes is of the form $B_{\us}$. This is shown in the following sequence
wherein 
we use the notation $[n]=\{1,2,\ldots,n\}$:
\begin{align*}
\rhom_R(B_{\ui},C_n)
&\simeq \rhom_R(B_{\ui},B_1\lotimes_R\cdots\lotimes_R B_n)
\simeq\rhom_R(B_{\ui},B_{[n]})
\simeq B_{[n]\smallsetminus\ui}
\end{align*}
The first isomorphism follows from Remark~\ref{disc0201}.5 using the 
assumption $C_0\simeq R$.
The second isomorphism is by definition.
The third isomorphism is from Lemma~\ref{lem9905}.
\end{proof}

The next result provides a purely combinatorial description of the reflexivity ordering on the set of 
$\eq{B_{\ui}}$.

\begin{thm} \label{thm0403}
Assume that 
$\catss_{C_1}(R)\subseteq\cdots \subseteq\catss_{C_{n-1}}(R) \subseteq\catss_{C_n}(R)$.
Under  the hypotheses of Assumption~\ref{ass},
the following conditions are equivalent:
\begin{enumerate}[\rm(i)]
\item\label{thm0403b1}
One has $B_{\ui}\tri B_{\us}$;
\item\label{thm0403b2}
One has $B_{\us}\in\catb_{B_{\ui}}(R)$; and
\item\label{thm0403b3}
One has $\ui\supseteq \us$.
\end{enumerate}
\end{thm}

\begin{proof}
The equivalence~\eqref{thm0403b1}$\iff$\eqref{thm0403b2} 
is from Remark~\ref{disc9801};
the implication~\eqref{thm0403b3}$\implies$\eqref{thm0403b1}
is contained in Lemma~\ref{lem9905}.

\eqref{thm0403b1}$\implies$\eqref{thm0403b3}:
Assume that 
$\eq{B_{\ui}}\tri \eq{B_{\us}}$. Lemma~\ref{lem9905} implies that 
$\eq{B_{\us\cup\ui}}\tri \eq{B_{\us}}$ and $\eq{B_{\ui}}\tri \eq{B_{\ui\smallsetminus\us}}$,
and furthermore that
$\rhom_R(B_{\us},B_{\us\cup\ui})\simeq B_{\ui\smallsetminus\us}$.
Since the reflexivity ordering on $\catss(R)$ is transitive, we have
$$\eq{B_{\us\cup\ui}}\tri\eq{B_{\us}}\tri\eq{\rhom_R(B_{\us},B_{\us\cup\ui})}.$$
From~\cite[(3.3.a)]{frankild:sdcms} we conclude that
$\eq{B_{\us\cup\ui}}=\eq{B_{\us}}$,
and hence the first equality in the next sequence:
$$\eq R=\eq{\rhom_R(B_{\us},B_{\us\cup\ui})}=\eq{B_{\ui\smallsetminus\us}}.$$
It is straightforward to show that this implies that $\ui\smallsetminus\us=\emptyset$,
and thus $\ui\subseteq\us$.
\end{proof}

\begin{cor} \label{cor0403}
Assume that 
$\catss_{C_1}(R)\subseteq\cdots \subseteq\catss_{C_{n-1}}(R) \subseteq\catss_{C_n}(R)$.
Under  the hypotheses of Assumption~\ref{ass},
one has $\eq{B_{\ui}}=\eq{B_{\us}}$ if and only if $\ui=\us$.
In particular, there are exactly $2^n$ classes of the form $\eq{B_{\ui}}$.
\end{cor}

\begin{proof}
One implication of the biconditional statement
is straightforward. For the converse, assume that 
$\eq{B_{\ui}}=\eq{B_{\us}}$. Then $\eq{B_{\ui}}\tri\eq{B_{\us}}\tri\eq{B_{\ui}}$,
so Theorem~\ref{thm0403} implies that $\ui\subseteq\us\subseteq\ui$,
that is $\ui=\us$ as desired. 

It follows that there are exactly $2^n$ classes of the form $\eq{B_{\ui}}$
because this is the number of possible $\ui\subseteq[n]$.
\end{proof}

\begin{disc} \label{d9901}
Assume that $n=3$ and that
$\catss_{C_1}(R)\subseteq\catss_{C_{2}}(R) \subseteq\catss_{C_3}(R)$.
The next diagram shows the case  of the 
ordered set of $R$-complexes of the form $B_{\ui}$, 
with all the reflexivity relations  indicated with edges:
$$
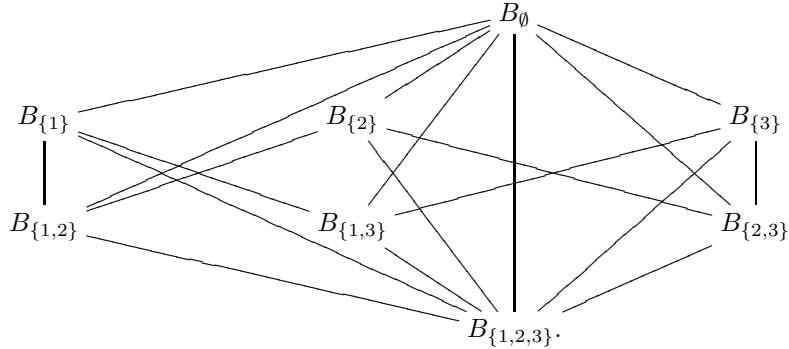
\xymatrix{
&&&&B_{\emptyset} \ar@{-}[lllld]\ar@{-}[ld]\ar@{-}[rrd]\ar@{-}[lllldd]\ar@{-}[ldd]\ar@{-}[rrdd]\ar@{-}[ddd]
\\
B_{\{1\}}\ar@{-}[d]\ar@{-}[drrr]\ar@{-}[ddrrrr]
&&&B_{\{2\}}\ar@{-}[llld]\ar@{-}[drrr]\ar@{-}[ddr]
&&&B_{\{3\}} \ar@{-}[llld]\ar@{-}[d]\ar@{-}[lldd]
\\
B_{\{1,2\}}\ar@{-}[rrrrd]
&&&B_{\{1,3\}}\ar@{-}[rd]
&&&B_{\{2,3\}} \ar@{-}[lld]
\\
&&&&B_{\{1,2,3\}}.
}$$
Compare this to the diagram in Remark~\ref{d1101}.
\end{disc}

We conclude with a version of Theorem~\ref{thm0403} for Auslander classes.

\begin{thm} \label{thm0403'}
Assume that 
$\catss_{C_1}(R)\subseteq\cdots \subseteq\catss_{C_{n-1}}(R) \subseteq\catss_{C_n}(R)$.
Under  the hypotheses of Assumption~\ref{ass},
the following conditions are equivalent:
\begin{enumerate}[\rm(i)]
\item\label{thm0403d1}
One has $B_{\ui}\in\cata_{B_{\us}}(R)$;
\item\label{thm0403d2}
One has $B_{\us}\in\cata_{B_{\ui}}(R)$;
\item\label{thm0403d3}
The $R$-complex $B_{\ui}\lotimes_RB_{\us}$ is semidualizing; and
\item\label{thm0403d4}
One has $\ui\cap\us=\emptyset$.
\end{enumerate}
\end{thm}

\begin{proof}
The equivalences~\eqref{thm0403d1}$\iff$\eqref{thm0403d2}$\iff$\eqref{thm0403d3}  
are from Remark~\ref{disc9801}.

\eqref{thm0403d4}$\implies$\eqref{thm0403d3}:
If $\ui\cap\us=\emptyset$, then
$B_{\ui}\lotimes_RB_{\us}\simeq B_{\ui\cup\us}$, which  
is semidualizing by Lemma~\ref{lem9904}.

\eqref{thm0403d3}$\implies$\eqref{thm0403d4}:
Assume that $B_{\ui}\lotimes_RB_{\us}$ is a semidualizing $R$-complex,
and suppose that $a\in\ui\cap\us$.
It follows that we have
$B_{\ui}\lotimes_RB_{\us}\simeq B_a\lotimes_RB_a\lotimes_RX$ where
$X=B_{\ui\smallsetminus\{a\}}\lotimes_RB_{\us\smallsetminus\{a\}}$.
Remark~\ref{disc0201}.4 implies that $B_a\sim R$,
and hence
\begin{align*}
C_{a-1}
&\simeq\rhom_R(\rhom_R(C_{a-1},C_a),C_a)\\
&\simeq\rhom_R(B_a,C_a)\\
&\sim\rhom_R(R,C_a)\\
&\simeq C_a.
\end{align*}
This contradicts the condition $\eq{C_a}\neq\eq{C_{a-1}}$
which is part of the assumption $\eq{C_a}\trineq\eq{C_{a-1}}$.
Thus, we must have $\ui\cap\us=\emptyset$.
\end{proof}



\providecommand{\bysame}{\leavevmode\hbox to3em{\hrulefill}\thinspace}
\providecommand{\MR}{\relax\ifhmode\unskip\space\fi MR }
\providecommand{\MRhref}[2]{%
  \href{http://www.ams.org/mathscinet-getitem?mr=#1}{#2}
}
\providecommand{\href}[2]{#2}

\end{document}